\documentclass[12pt]{amsart}
\usepackage{amsmath,amssymb}
\usepackage{amsfonts}
\usepackage{amsthm}
\usepackage{latexsym}
\usepackage{graphicx}

\def\p{\partial}

\def\R{\mathbb{R}}

\def\vv<#1>{\langle#1\rangle}
\def\ol{\overline}

\def\Div{{\rm div}}
\def\Span{{\rm span}}

\def\XXint#1#2{\setbox0=\hbox{$#1{#2}{\int}$}{#2}\kern-.5\wd0 }

\def\XXint#1#2#3{{\setbox0=\hbox{$#1{#2#3}{\int}$}
     \vcenter{\hbox{$#2#3$}}\kern-.5\wd0}}

\def\vv<#1>{{\left\langle#1\right\rangle}}

\def\Vol{\mbox{Vol}}

\def\Int{{\rm Int}}
\newtheorem{thm}{Theorem}[section]

\newtheorem{cor}{Corollary}[section]
\theoremstyle{definition}

\theoremstyle{remark}

\numberwithin{equation}{section}

\begin{document}
\title{Rigidity of a trace estimate for Steklov eigenvalues}

\author{Yongjie Shi$^1$}
\address{Department of Mathematics, Shantou University, Shantou, Guangdong, 515063, China}
\email{yjshi@stu.edu.cn}
\author{Chengjie Yu$^2$}
\address{Department of Mathematics, Shantou University, Shantou, Guangdong, 515063, China}
\email{cjyu@stu.edu.cn}
\thanks{$^1$Research partially supported by NSF of China with contract no. 11701355. }
\thanks{$^2$Research partially supported by NSF of China with contract no. 11571215.}
\renewcommand{\subjclassname}{%
  \textup{2010} Mathematics Subject Classification}
\date{}
\begin{abstract}
In this short note, we show the rigidity of a trace estimate for Steklov eigenvalues with respect to functions in our previous work (Trace and inverse trace of Steklov eigenvalues. J. Differential Equations 261 (2016), no. 3, 2026--2040.). Namely, we show that equality of the estimate holds  if and only if the manifold is a direct product of a round ball and a closed manifold. The key ingredient in the proof is a decomposition theorem for flat and totally geodesic Riemannian submersions which may be of independent interests.
\end{abstract}
\maketitle\markboth{Shi \& Yu}{Rigidity of a trace estimate}
\section{Introduction}
Let $(M^n,g)$ be a compact  Riemannian manifold with nonempty boundary. If the following boundary value problem:
\begin{equation}
\left\{\begin{array}{l}\Delta u=0\\
\frac{\p u}{\p\nu}=\sigma u
\end{array}\right.
\end{equation}
has a nontrivial solution, then we call the constant $\sigma$ a Steklov eigenvalue of $(M,g)$. Here $\nu$ is the unit outward normal vector field on $\p M$. The Rayleigh quotient corresponding to Steklov eigenvalues is 
\begin{equation}
Q(u)=\frac{\int_M\|\nabla u\|^2dV_M}{\int_{\p M}u^2dV_{\p M}}.
\end{equation}

Steklov \cite{St,Ku} considered this kind of eigenvalue problems because it is closely related to the frequency of liquid sloshing in a container. It is not hard to see that Steklov eigenvalues are just eigenvalues of the Dirichlet-to-Neumann map that sends Dirichlet boundary data of a harmonic function on $M$ to its Neumann boundary data. Steklov eigenvalues were extensively study in the past decades, because it is deeply related to free boundary minimal submanifolds and conformal geometry in differential geometry (\cite{FS1,FS2}), liquid sloshing in  physics and Calder\'on inverse problem (\cite{Ca,Ul}) in applied mathematics.

Higher order Dirichlet-to-Neumann maps were also considered in literature \cite{BS,JL}, because they are closely related to inverse problems for the Maxwell equation in electromagnetics. However, the Dirichlet-to-Neumann maps considered in \cite{BS,JL} was not suitable for spectral analysis. In 2012, Raulot and Savo \cite{RS2} introduced a new notion of higher order Dirichlet-to-Neumann maps which is suitable for spectral analysis.

We would also like to mention that discrete versions of the classical Dirichlet-to-Neumann maps and higher order Dirichlet-to-Neuman maps were introduced in \cite{HHW} and \cite{SY2} respectively.

The Steklov eigenvalues of a Riemannian manifold $(M,g)$ can be listed in ascending order counting multiplicity as follows:
\begin{equation*}
0=\sigma_0<\sigma_1\leq\sigma_2\leq\cdots\leq\sigma_k\leq\cdots.
\end{equation*}
In \cite{SY1}, by further extending the idea of Raulot-Savo in \cite{RS1,RS2}, among the others, we obtained the following trace estimate for Steklov eigenvalues:
\begin{thm}\label{thm-trace}
Let $(M^n,g)$ be a compact Riemannian manifold with nonempty boundary and $V$ be the space of parallel exact 1-forms on $M$. Suppose that $\dim V=m>0$. Then
\begin{equation}\label{eq-trace-0}
\sigma_1+\sigma_2+\cdots+\sigma_m\leq \frac{\Vol(\p M)}{\Vol (M)}.
\end{equation}
\end{thm}
As a direct corollary of Theorem \ref{thm-trace}, we have
\begin{equation}\label{eq-trace-eu}
\sigma_1+\sigma_2+\cdots+\sigma_n\leq \frac{\Vol(\p\Omega)}{\Vol(\Omega)}
\end{equation}
for any bounded smooth domain $\Omega$ in $\R^n$. This estimate is sharp because the equality holds when $\Omega$ is a round ball. By using the Cauchy-Schwartz inequality, one has
\begin{equation}
\frac{1}{\sigma_1}+\frac{1}{\sigma_2}+\cdots+\frac{1}{\sigma_n}\geq\frac{n^2\Vol(\Omega)}{\Vol(\p \Omega)}.
\end{equation}
However, this estimate is weaker than Brock's inverse trace estimate \cite{Br}:
\begin{equation}
\frac{1}{\sigma_1}+\frac{1}{\sigma_2}+\cdots+\frac{1}{\sigma_n}\geq \frac{n\Vol^\frac1n(\Omega)}{\Vol^\frac1n(\mathbb{B}^n)}
\end{equation}
because of the isoperimetric inequality for bounded Euclidean domains.

In this paper, we characterize the equality case of \eqref{eq-trace-0}. In summary, combining with Theorem \ref{thm-trace}, we have the following result.
\begin{thm}\label{thm-trace-rigidity}
Let $(M^n,g)$ be a compact Riemannian manifold with nonempty boundary and $V$ be the space of parallel exact 1-forms on $M$. Suppose that $\dim V=m>0$. Then
\begin{equation}\label{eq-trace}
\sigma_1+\sigma_2+\cdots+\sigma_m\leq \frac{\Vol(\p M)}{\Vol (M)}.
\end{equation}
The equality holds if and only if $M$ is a metric product of $\mathbb{B}^m(R)$ and a closed manifold $F$ with $R=\displaystyle\frac{m\Vol(M)}{\Vol(\p M)}$ and 
\begin{equation}\label{eq-sigma-mu}
\sigma(\mu_1(F))\geq\frac{1}{R}
\end{equation}
where $\mu_1(F)$ is the first positive eigenvalue for the Laplacian operator on $F$ and 
\begin{equation}
\sigma(\mu):=\inf_{f\in C^\infty(\mathbb{B}^m(R))}\frac{\int_{\mathbb{B}^m(R)}(\|\nabla f\|^2+\mu f^2)dV_{\mathbb{B}^m(R)}}{\int_{\p{\mathbb{B}^m(R)}}f^2dV_{\p{\mathbb{B}^m(R)}}}
\end{equation}
which is the first eigenvalue of the following boundary value problem:
\begin{equation}
\left\{\begin{array}{ll}\Delta f=\mu f& {\rm on}\  {\mathbb{B}^m(R)}\\
\frac{\p f}{\p \nu}=\sigma f&{\rm on }\ \p{\mathbb{B}^m(R)}.
\end{array}\right.
\end{equation}
\end{thm}
The restriction \eqref{eq-sigma-mu} is necessary. For example, consider $M=[-1,1]\times \mathbb S^1(L)$ where $\mathbb S^1(L)$ is the round circle with radius $L$. As computed in \cite{FS1}, 
\begin{equation}
\sigma_1(M)=\min\left\{\sigma\left(\frac{1}{L^2}\right),1\right\}=\min\left\{\frac{1}{L}\tanh\left(\frac1L\right),1\right\}.
\end{equation} 
So, 
\begin{equation}
\sigma_1(M)=\frac{\Vol (\p M)}{\Vol(M)}=1
\end{equation}
only when $\frac{1}{L}\tanh\left(\frac1L\right)\geq 1$.

Moreover, as a direct corollary of Theorem \ref{thm-trace-rigidity}, one has
\begin{cor}Let $\Omega\subset \R^n$ be a bounded domain with smooth boundary. Then,
\begin{equation}
\sigma_1+\sigma_2+\cdots+\sigma_n\leq \frac{\Vol(\p \Omega)}{\Vol(\Omega)}.
\end{equation}
The equality holds if and only if $\Omega$ is a round ball.
\end{cor}

The key ingredient in the proof of the rigidity in Theorem \ref{thm-trace-rigidity} is the following result about triviality of a flat and totally geodesic Riemannian submersion between two complete Riemannian manifolds with boundary. Here a Riemanian submersion $\pi:M\to N$ is said to be flat if the horizontal distribution is integrable, and is said to be totally geodesic if each fibre is totally geodesic (see \cite{Wa}).
\begin{thm}\label{thm-decompose}
Let $\pi:(\ol M^{n+r},\ol g)\to (M^n,g)$ be a flat and totally geodesic Riemannian submersion between two complete Riemannian manifolds with boundary such that $\pi(\p\ol M)= \p M$. Suppose that $M$ is simply connected. Then, there is a complete Riemanian manifold $F$ without boundary and an isometry $\varphi:\ol M\to F\times M$ such that $\pi_M\circ\varphi=\pi$.
\end{thm}

Note that the result fails when $M$ is not simply connected. For example, let $M$ be the annulus $\{x\in \R^2\ |\  1\leq\|x\|\leq 2\}$ with standard metric and $\ol M$ be its universal cover, then the conclusion fails. Moreover, it is not hard to see that the horizontal and vertical distributions of a flat and totally geodesic submersion are both parallel (see \cite{Wa}). The assumption $\pi(\p\ol M)=\p M$ implies that normal vectors on $\p \ol M$ must be horizontal. So, by the deRham decompositions for Riemannian manifolds with boundary by the second named author \cite{Yu}, we have the splitting conclusion when $\ol M$ is simply connected. Theorem \ref{thm-decompose} just gives us a splitting conclusion by replacing the simply connectedness of $\ol M$ by the simply connectedness of $M$ which is more suitable for our application in the proof of Theorem \ref{thm-trace-rigidity}.

\section{Proofs of main results}
Let's first recall the notion of development which will be used the in the proof of Theorem \ref{thm-decompose}. Let $v:[0,T]\to T_pM$ be a curve in $T_pM$. A curve $\gamma:[0,T]\to M$ with
\begin{equation}
\gamma'(t)=P_0^t(\gamma)(v(t))\ \mbox{and}\ \gamma(0)=p
\end{equation}
is called the development of $v$. This notion was presented in the language of principal fibre bundle in \cite{KN}. A proof of the local existence and uniqueness of developments can be found in \cite{Yu}.
\begin{proof}[Proof of Theorem \ref{thm-decompose}]
Let $p\in \Int M$ be a fixed point and $F=\pi^{-1}(p)$ where $\Int M=M\setminus\p M$. Because, $\pi(\p \ol M)=\p M$, $F$ is a complete Riemanian manifold without boundary. For each $q\in M$, let $\gamma:[0,1]\to M$ be a curve such that $\gamma(0)=p$ and $\gamma(1)=q$. For each $\ol p\in F$, let $\ol\gamma_{\ol p}$ be the horizontal lift of $\gamma$ with $\ol\gamma_{\ol p}(0)=\ol p$. Consider the map $\Phi_\gamma:F\to \pi^{-1}(q)$ sending $\ol p\to \ol\gamma_{\ol p}(1)$. It is not hard to see that $\Phi_\gamma$ is an isometry (see \cite{Ne}). To show the conclusion of the theorem, we only need to prove that $\Phi_\gamma$ is independent of the choice of $\gamma$ and depending only on the end point $q$ (see \cite[Thoerem 5]{Ne}).

Let $\gamma_0,\gamma_1:[0,1]\to M$ be two smooth  curves with $\gamma_0(0)=\gamma_1(0)=p$ and $\gamma_0(1)=\gamma_1(1)=q$. Because $M$ is simply connected, there is a homotopy $\Psi:[0,1]\times[0,1]\to M$ such that
\begin{equation}
\left\{\begin{array}{l}\Psi(0,t)=\gamma_0(t)\\
\Psi(1,t)=\gamma_1(t)\\
\Psi(u,0)=p\\
\Psi(u,1)=q.\\
\end{array}\right.
\end{equation}
For each $\ol p\in F$, let $\ol\Psi_{\ol p}$ be the horizontal lift of $\Psi$ with $\ol\Psi_{\ol p}(u,0)=\ol p$. Let $v(u,t)\in T_pM$ be given by
\begin{equation}
v(u,t)=P_t^0(\gamma_u)(\gamma_u'(t))
\end{equation}
where $\gamma_u(t)=\Psi(u,t)$. Then, $\gamma_u$ is the development of $v(u,\cdot)$. Let $\ol v(u,t)\in T_{\ol p}\ol M$ be the horizontal lift of $v(u,t)$. Because $\pi:\ol M\to M$ is locally splitting (see \cite{Wa}), $\ol\Psi_{\ol p}(u,\cdot)$ is the development of $\ol v(u,\cdot)$.

Let $e_1,e_2,\cdots,e_n$ be an orthonormal frame of $T_pM$ and $\ol e_i\in T_{\ol p}\ol M$ be the horizontal lift of $e_i$ for $i=1,2,\cdots, n$. Let $E_i(u,t)=P_{0}^t(\gamma_u)(e_i)$ and $\ol E_i(u,t)=P_0^t(\ol\gamma_u)(\ol e_i)$ where $\ol\gamma_u(t)=\ol\Psi_{\ol p}(u,t)$. Because $\pi$ is locally splitting again, $\ol E_i(u,t)$ is the horizontal lift of $E_i(u,t)$. Let $\ol e_{n+1},\cdots ,\ol e_{n+r}\in T_{\ol p}F$ be an orthonormal frame and $\ol E_\alpha(u,t)=P_{0}^t(\ol\gamma_u)(\ol e_\alpha)$ for $\alpha=n+1,\cdots,n+r$. By parallel-ness of the vertical distribution, $\ol E_\alpha(u,t)$ is vertical.

Suppose that
\begin{equation}
\frac{\p \Psi}{\p u}=\sum_{i=1}^nU_iE_i
\end{equation}
and
\begin{equation}
v=\sum_{i=1}^nv_ie_i.
\end{equation}
Then, by \cite{Yu}, $U_i$'s satisfy the following Cauchy problem:
\begin{equation}\label{eq-g-Jacobi}
\begin{split}
\left\{\begin{array}{ll}U''_i=\displaystyle\sum_{j,k,l=1}^nv_kv_lR(E_k,E_i,E_l,E_j)U_j+\p_u\p_tv_i+\sum_{j=1}^n\p_tv_jX_{ji}&i=1,2,\cdots,n\\
X'_{ij}=\displaystyle\sum_{k,l=1}^nv_lR(E_i,E_j,E_l,E_k)U_k&i,j=1,2,\cdots,n\\
X_{ij}(u,0)=0&i,j=1,2,\cdots,n\\
U_i(u,0)=0&i=1,2,\cdots,n\\
U'_i(u,0)=\p_uv_i(u,0)&i=1,2,\cdots,n.
\end{array}\right.
\end{split}
\end{equation}
Here the symbol $'$ means taking derivative with respect to $t$ and $R$ is the curvature tensor of $M$.

Suppose that
\begin{equation}
\frac{\p \ol\Psi_{\ol p}}{\p u}=\sum_{a=1}^{n+r}\ol U_a\ol E_a,
\end{equation}
and note that
\begin{equation}
\ol v=\sum_{i=1}^nv_i\ol e_i.
\end{equation}
So, by \cite{Yu} again, $\ol U_a$'s satisfy the following Cauchy problem:
\begin{equation}\label{eq-g-Jacobi-2}
\begin{split}
\left\{\begin{array}{ll}\ol U''_a=\displaystyle\sum_{b,c,d=1}^{n+r}v_cv_d\ol R(\ol E_c,\ol E_a,\ol E_d,\ol E_b)\ol U_b+\p_u\p_tv_a+\sum_{b=1}^{n+r}\p_tv_b\ol X_{ba}&a=1,2,\cdots,n+r\\
\ol X'_{ab}=\displaystyle\sum_{c,d=1}^{n+r}v_d\ol R(\ol E_a,\ol E_b,\ol E_d,\ol E_c)U_c&a,b=1,2,\cdots,n+r\\
\ol X_{ab}(u,0)=0&a,b=1,2,\cdots,n+r\\
\ol U_a(u,0)=0&a=1,2,\cdots,n+r\\
\ol U'_a(u,0)=\p_uv_a(u,0)&a=1,2,\cdots,n+r\\
\end{array}\right.
\end{split}
\end{equation}
Here $\ol R$ is curvature tensor of $\ol M$ and we take $v_\alpha=0$ for $\alpha=n+1,n+2,\cdots,n+r$.

Because $\pi$ is local splitting,
\begin{equation}
\ol R(\ol E_i,\ol E_j,\ol E_k,\ol E_l)=R(E_i,E_j,E_k,E_l),
\end{equation}
\begin{equation}
\ol R(\ol E_\alpha,\ol E_j,\ol E_k,\ol E_l)=0
\end{equation}
and
\begin{equation}
\ol R(\ol E_\alpha,\ol E_\beta,\ol E_k,\ol E_l)=0
\end{equation}
for $i,j,k,l=1,2,\cdots,n$ and $\alpha,\beta=n+1,\cdots,n+r$. So, it is not hard to check that
\begin{equation}
\begin{split}
\left\{\begin{array}{l}\ol U_i=U_i\\
\ol U_\alpha=0\\
\ol X_{ij}=X_{ij}\\
\ol X_{i\alpha}=\ol X_{\alpha i}=\ol X_{\alpha\beta}=0
\end{array}\right.
\end{split}
\end{equation}
is the solution of the Cauchy problem \eqref{eq-g-Jacobi-2}. Therefore,
\begin{equation}
\frac{\p \ol\Psi_{\ol p}}{\p u}(u,1)=0
\end{equation}
and $\Phi_{\gamma_0}=\Phi_{\gamma_1}$. This completes the proof of the theorem.
\end{proof}

We are now ready to prove Theorem \ref{thm-trace-rigidity}. For completeness, we will also present a proof of the estimate \eqref{eq-trace} which uses the idea of Raulot and Savo in \cite{RS1,RS2} and is not the same as our previous proof in \cite{SY1}.

\begin{proof}[Proof of Theorem \ref{thm-trace-rigidity}]
Note that for any harmonic function $f$ on $M$, one has
\begin{equation}
\begin{split}
\left(\int_{M}\|\nabla f\|^2dV_M\right)^2=&\left(\int_{ M}\Div(f\nabla f)dV_M\right)^2\\
=&\left(\int_{\p M}f\frac{\p f}{\p\nu}dV_{\p M}\right)^2\\
\leq&\int_{\p M}f^2dV_{\p M}\int_{\p M}\left(\frac{\p f}{\p \nu}\right)^2dV_{\p M}.
\end{split}
\end{equation}
Equality holds if and only if $f$ is a Steklov eigenfunction. So, for any non-trivial  harmonic function $f$ on $M$,
\begin{equation}\label{eq-key}
\frac{\int_{M}\|\nabla f\|^2dV_{M}}{\int_{\p M}f^2dV_{\p M}}\leq \frac{\int_{\p M}\left(\frac{\p f}{\p\nu}\right)^2dV_{\p M}}{\int_M\|\nabla f\|^2dV_M}.
\end{equation}
Let
$$U=\left\{f\in C^\infty(M)\ \bigg|\ \int_{\p M}fdV_{\p M}=0,\ df\mbox{ is paralell.}\right\}.$$
By assumption, we know that $\dim U=m$. By the inequality \eqref{eq-key} and Courant's min-max principle, we have
\begin{equation}
\sigma_k\leq \lambda_k
\end{equation}
for any $k=1,2,\cdots,m$, where $\lambda_1\leq \lambda_2\leq\cdots\leq\lambda_m$ are the eigenvalues of the Rayleigh quotient $\frac{\int_{\p M}\left(\frac{\p f}{\p\nu}\right)^2dV_{\p M}}{\int_M\|\nabla f\|^2dV_M}$ restricted on $U$.

Let $f_1,f_2,\cdots, f_m$ be a basis of $U$ such that
\begin{equation}
\vv<\nabla f_i,\nabla f_j>\equiv\delta_{ij}
\end{equation}
and  $f_i$ is the eigenfunction of the Rayleigh quotient $\frac{\int_{\p M}\left(\frac{\p f}{\p\nu}\right)^2dV_{\p M}}{\int_M\|\nabla f\|^2dV_M}$ restricted on $U$ with respect to the eigenvalue $\lambda_i$. Then
\begin{equation}
\begin{split}
\sigma_1+\sigma_2+\cdots+\sigma_m\leq& \lambda_1+\lambda_2+\cdots+\lambda_m\\
=&\frac{\int_{\p M}\sum_{i=1}^m\left(\frac{\p f_i}{\p\nu}\right)^2dV_{\p M}}{\Vol(M)}\\
=&\frac{\int_{\p M}\sum_{i=1}^m\vv<\nu,\nabla f_i>^2dV_{\p M}}{\Vol(M)}\\
\leq&\frac{\int_{\p M}\|v\|^2dV_{\p M}}{\Vol(M)}\\
=&\frac{\Vol(\p M)}{\Vol(M)}.
\end{split}
\end{equation}
When the equality holds, one has $\nu\in \Span\{\nabla f_1,\nabla f_2,\cdots,\nabla f_m\}$ and $f_1,f_2,\cdots,f_m$ are the Steklov eigenfunctions with respect to $\sigma_1,\sigma_2,\cdots,\sigma_m$ respectively. So,
\begin{equation}\label{eq-nu-1}
\begin{split}
\nu=\sum_{i=1}^m\vv<\nu,\nabla f_i>\nabla f_i=\sum_{i=1}^m\sigma_i f_i \nabla f_i.
\end{split}
\end{equation}
Therefore,
\begin{equation}\label{eq-sum-square}
\sum_{i=1}^m\sigma_i^2f_i^2=\|\nu\|^2=1
\end{equation}
on $\p M$. This implies that $\nu$ must be parallel to
\begin{equation}\label{eq-nu-2}
\nabla\left(\sum_{i=1}^m\sigma_i^2f_i^2\right)=2\sum_{i=1}^m\sigma_i^2f_i\nabla f_i.
\end{equation}
By comparing \eqref{eq-nu-1} and \eqref{eq-nu-2}, we have
\begin{equation}
\sigma_1=\sigma_2=\cdots=\sigma_m=\frac{\Vol(\p M)}{m\Vol(M)}.
\end{equation}
Then, by \eqref{eq-sum-square},
\begin{equation}
\sum_{i=1}^mf_i^2=\left(\frac{m\Vol(M)}{\Vol(\p M)}\right)^2=R^2
\end{equation}
on $\p M$. Let $\pi=(f_1,f_2,\cdots,f_m)$. It is not hard to check that $\pi$ is a Riemannian submersion from $M$ to $\mathbb{B}^m(R)$ with $\pi(\p M)=\p\mathbb{B}^m(R)$. By Theorem \ref{thm-decompose}, we know that $M=\mathbb B^m(R)\times F$ for some closed Riemannian manifold $F$. 

Moreover, when $M=\mathbb B^m(R)\times F$ with $F$ a closed Riemannian manifold. Let $$0=\mu_0<\mu_1\leq \mu_2\leq\cdots\leq \mu_k\leq\cdots$$ be the spectrum of the Laplacian operator for $F$. Then, by a standard argument as in \cite{FS1} using the method of separating variables, we know that the Steklov spectrum of $M$ is formed by the eigenvalues $\sigma$ of the following boundary value problems on $\mathbb B^m(R)$:
\begin{equation}\label{eq-g-boundary}
\left\{\begin{array}{ll}\Delta f=\mu_i f&{\rm on}\ \mathbb B^m(R),\\
\frac{\p f}{\p \nu}=\sigma f&{\rm on}\ \p \mathbb B^m(R),
\end{array}\right.
\end{equation} 
for $i=0,1,2,\cdots.$ The corresponding Rayleigh quotient of the eigenvalue problem \eqref{eq-g-boundary} is 
\begin{equation}
Q_i(f)=\frac{\int_M(\|\nabla f\|^2+\mu_i f^2)dV_M}{\int_{\p M}f^2dV_{\p M}}
\end{equation}
for $i=0,1,2,\cdots$. So, it is not hard to see that 
\begin{equation}
\sigma_1(M)=\cdots=\sigma_m(M)=\sigma_1(\mathbb B^m(R))=\cdots=\sigma_m(\mathbb B^m(R))=\frac1R
\end{equation}
which is the first $m$ positive Steklov eigenvalues of $M$ only when
\begin{equation}
\sigma(\mu_1)\geq \frac{1}{R}
\end{equation}
because 
\begin{equation}
\sigma(\mu_1)\leq\sigma(\mu_2)\leq\cdots\leq\sigma(\mu_k)\leq\cdots.
\end{equation}
This completes the proof of the theorem.
\end{proof}

\end{document}